\documentclass[a4paper,10pt]{article}

\usepackage[latin1]{inputenc}
\usepackage[T1]{fontenc}
\usepackage[english]{babel}

\usepackage{mathrsfs}
\usepackage{amscd}
\usepackage{amsfonts}
\usepackage{amsmath}
\usepackage{amssymb}
\usepackage{amstext}
\usepackage{amsthm}
\usepackage{amsbsy}

\usepackage{xspace}
\usepackage[all]{xy}
\usepackage{graphicx}
\usepackage{url}
\usepackage{latexsym}

\usepackage{graphicx} 

\usepackage{booktabs} 
\usepackage{array} 
\usepackage{paralist} 
\usepackage{verbatim} 
\usepackage{subfig} 

\usepackage{fancyhdr} 
\usepackage{sectsty}
\allsectionsfont{\sffamily\mdseries\upshape} 

\usepackage[nottoc,notlof,notlot]{tocbibind} 
\usepackage[titles,subfigure]{tocloft} 

\usepackage[textwidth=100pt,textsize=footnotesize,bordercolor=white,color=blue!30]{todonotes}
\usepackage{hyperref} 

\makeatletter
\newcommand*{\rom}[1]{\expandafter\@slowromancap\romannumeral #1@}
\makeatother

\theoremstyle{definition}

\newtheorem{fact}{fact}

\newtheorem{thm}[fact]{Theorem}
\newtheorem{lemma}[fact]{Lemma}
\newtheorem{prop}[fact]{Proposition}

\newtheorem{defini}[fact]{Definition}

\title{Generalized Effective Reducibility}
\author{Merlin Carl}
\date{}

\begin{document}

\maketitle

\begin{abstract}
We introduce two notions of effective reducibility for set-theoretical statements, based on computability with Ordinal Turing Machines (OTMs), one of which resembles Turing reducibility while the other is modelled after Weihrauch reducibility.
We give sample applications by showing that certain (algebraic) constructions are not effective in the OTM-sense and considerung the effective equivalence of various versions of the axiom of choice.
\end{abstract}

\section{Introduction}


From a sufficiently remote point of view, construction problems in mathematics can be seen as multi-valued, class-sized `functions' from the
set-theoretical universe $V$ to itself. Example of construction problems would be the problem assigning to fields their algebraic closures,
to sets their well-orderings, to integrable functions their stem functions, to linear orderings their completions etc. Formally, this makes a construction problem
a (class-sized) relation $R\subseteq V\times V$.

A `solution' to or `canonification' of a construction problem $R$ is then a (class-sized) witness `function' $F:V\rightarrow V$ such that, for all $x$ in the domain of $R$,
we have $R(x,F(x))$ and otherwise $F(x)=\emptyset$. Similarly, we can say that $F$ witnesses the truth of a set-theoretical statement $\phi$ of the form $\forall{x}\exists{y}\psi$
if $F$ is a solution for $\{(x,y):\psi(x,y)\}$, the most natural candidates to consider being $\Pi_{2}$-statements, since $\psi$ can be assumed to be absolute between transitive sets
in that case.

Fixing an appropriate notion of effectiveness for set-theoretical constructions, we can now ask for specific construction problems $R$ whether there exists an \textbf{effective} solution for $R$ and similarly,
whether some statement $\phi$ is `effectively true'. Moreover, we can ask whether a construction or a statement `effectively reduces' to another.

In the following, `effectiveness' will be interpreted to mean computability by Ordinal Turing Machines (OTMs) without ordinal parameters. It was argued in \cite{Ca} that OTM-computations are appropriate as a formalization
of the intuitive notion of a `transfinite effective procedure'. One indication is the equivalence of this with various other `maximal' models of ordinal computability, such as Ordinal Register Machines \cite{ORM}
or ordinal $\lambda$-calculus (\cite{Sey}, \cite{Fi}).

The definition and basic results on OTMs can be found in \cite{Ko1}. We merely briefly recall the model here: An OTM-program is just a normal Turing machine program with the usual (finite) set of commands
for reading and writing symbols, moving the read/write-head and changing the inner state. We assume that the inner states are indexed with natural numbers. 
The `hardware' of an OTM consists of a class-sized tape with cells indexed by ordinals. Each cell may contain
a $0$ or a $1$. The working time of an OTM is again the whole class of ordinals. At successor times, an OTM behaves like an ordinary Turing machine. At limit times, the head position, the inner state and the content of 
the $\iota$-th cell for each $\iota\in\text{On}$ are determined as the inferior limit of the sequence of earlier head positions, and inner states and contents of the $\iota$-th cell. If the 
read/write-head is asked to move to the left while currently occupying a cell with limit index, it is reset to the $0$th cell.

For convenience, we assume that our machines work with three tapes, a `miracle' tape (to be explained below), a scratch tape and an output tape. 
The single-tape model can easily be adapted to this setting.










\section{Basic Methods and Notions}

Our goal is to apply OTM-computability to general mathematical constructions.
To make this approach work, we need a way to represent arbitrary sets in a way suitable as an input format for OTMs. OTMs work on
a class-sized tape indexed with ordinals; a set $x$ will hence need to be represented as a set of ordinals. This can be achieved in a rather straightforward manner:

\begin{defini}
Let $x$ be a set, $t=\text{tc}(x)$ the transitive closure of $x$, $\alpha\in\text{On}$ and $f:\alpha\rightarrow\text{tc}(x)$ a well-ordering of $\text{tc}(x)$
in the order type $\alpha$. We define $c_{f}(x)$, the $f$-code for $x$, recursively as the following set or ordinals:
$c_{f}(x):=\{p(f^{-1}(y),\beta):y\in x\wedge \beta\in c_{f|y}(y)\}$, where $p$ denotes Cantor's ordinal pairing function. We say that $A\subseteq\text{On}$
`is a code for' or `codes' the set $x$ if and only if there is some $f$ for which $A=c_{f}(x)$. We write rep($\tau,x$) to indicate that $\tau$ codes $x$.
\end{defini}

\textbf{Remark}: By a certain abuse of notation, if $x$ is a set, we will sometimes write $c(x)$ for an `arbitrary' code for $x$.

We can now talk about OTM-computability of arbitrary functions from $V$ to $V$:

\begin{defini}
Let $F:V\rightarrow V$ be a functional class. We say that $F$ is OTM-computable if and only if there is an OTM-program $P$ such that, for every set $x$ and every tape content $\tau$, if rep$(\tau,x)$,
then $P(\tau)$ converges to output $\sigma$ such that rep$(\sigma,F(x))$, i.e. $P$ takes representations of $x$ to representations of $F(x)$.
\end{defini}

By this definition, the representation of a set $x$ will depend on the choice of a well-ordering of $\text{tc}(x)$. The output of a computation on input $x$
may hence depend on the choice of the representation of $x$. This is fine as long as only the output, but not the object coded by the output, depends on the
choice of the input representation. 

This allows us to make our notion of `effectivity' precise:

\begin{defini}
Let $R\subseteq V\times V$ be a construction problem. Then $R$ is effectively solvable if and only if there is an OTM-computable solution $F$ for $R$.
Moreover, a set-theoretical $\Pi_{2}$-statement $\forall{x}\exists{y}\phi(x,y)$ (where $\phi$ is $\Delta_{0}$) is effective if and only if
the construction problem $\{(x,y)\in V\times V:\phi(x,y)\}$ is effectively solvable. We write $R_{x}$ for $\{y:(x,y)\in R\}$.
\end{defini}

One may now inquire whether various well-known construction problems and $\Pi_{2}$-statements are effective. Such questions were studied by Hodges in \cite{Ho2}, though with
a different notion of effectivity based on Jensen and Karps primitive recursive set functions. We note here that the two methods Hodges uses also work for our model, which allows us to carry over
results.

The following lemma corresponds to Hodges' `cardinality method', i.e. Lemma 3.2 of \cite{Ho2}:

\begin{lemma}{\label{CardMethod}}
 Let $\alpha\in On$, and let $R\subseteq V\times V$ be such that, for some cardinal $\kappa>\alpha$,
there is $x\in V$ such that $|x|=\kappa$, $R_{x}\neq\emptyset$ and $\forall{y\in R_{x}}\text{card}(y)>\kappa$.
Then no witness function for $R$ is $OTM$-computable in the parameter $\alpha$.\\
Consequently, if $R$ is such that there are such $\kappa$ and $x$ for every $\alpha\in On$, then
no witness function for $R$ is parameter-$OTM$-computable.\\
In particular, if, for some $x$ of infinite cardinality, $R_{x}\neq\emptyset$
and $\forall{y\in R_{x}}\text{card}(y)>\text{card}(x)$ then no witness function for $R$ is parameter-free $OTM$-computable.
\end{lemma}
\begin{proof}
Clearly, in less then $\kappa^{+}$ many steps, the machine cannot write a code of a structure of cardinality $>\kappa$.

It hence suffices to show that, when $P$ is an OTM-program and $P$ is given a (code $c$ of a) set $x$ of size $\kappa\geq\omega$ for input and the computation halts, then the output of the computation will be of size $\leq\kappa$.
This follows if we can show that the computation will take less than $\kappa^{+}$ many steps, since $P$ can write at most $\alpha$ many symbols in $\alpha$ many steps. Suppose for a contradiction
that $P$ takes $\lambda>\kappa$ many steps, and let $\delta$ be the smallest cardinal $>\lambda$. Let $H$ be the $\Sigma_{1}$-Skolem hull of $\kappa\cup\{c\}$ in $L_{\delta}[c]$ and let $M$ denote the transitive collapse of $H$.
We may assume without loss of generality that $c\subseteq\kappa$, so that we have $c\in M$; as $L_{\delta}[c]$ contains the computation of $P$ in the input $c$, so does $H$ and hence there is $S\in M$ such that
$M$ believes that $S$ is the computation of $P$ with input $c$. By transitivity of $M$ and absoluteness of computations, $S$ is actually the computation of $P$ with input $c$. Since $S$ is contained in a transitive
set of cardinality $\kappa$, $|S|\leq\kappa$, so the length of the computation is $<\kappa^{+}$, as desired.

\end{proof}

There is also an analogue of the `forcing method' (Lemma 3.7 of \cite{Ho2}), which is given in Lemma \ref{reducibility} below. 

\smallskip

\textbf{Convention}: For many of the following results, we will need the existence of generic filters for various partial orderings in $L$ and some of its (symmetric) extensions. To avoid technical complications,
we use as a shortcut an extra assumption that guarantees the existence of such filters. $0^{\sharp}$ is more than enough for our purposes, and we assume from now on that it exists.\footnote{For some of the following
results, this assumption is actually necessary: It is e.g. not hard to check that all choice principles considered in section $4$ are effective (and hence trivially reducible to each other) if $V=L$.}

\smallskip



These lemmata can be seen as expressing the intuition that neither the power set operation on infinite set nor the use of the axiom of choice are `effective', not
even in a very idealized sense. We note some sample applications.

\begin{lemma}{\label{someconstructions}}
None of the following construction problems is effectively solvable:
\begin{enumerate}
\item Field to its algebraic closure
\item Linear ordering to its completions
\item Set to its (constructible) power set 
\item Set to its well-orderings
\end{enumerate}
\end{lemma}

\begin{proof}
(1) can be proved by an easy adaption of the proof of Theorem $4.1$ of \cite{Ho2}. There is only one point that requires a little care, namely the use of countable transitive models in that proof: For it might happen that
an OTM-program $P$ that halts in $V$ does not halt in such a model $M$.\footnote{For example, suppose there is some minimal countable $\alpha$ such that $L_{\alpha}\models\text{ZFC}$. Then the OTM-program that writes $L$ on the tape
until an $L$-level satisfying ZFC will halt in $V$, but not inside $L_{\alpha}$.} However, a check of Hodge's proof reveals that the countability of the ground model serves no purpose but to guarantee the existence of generic filters. 
We can hence circumvent this problem by doing the construction over $L$, using $0^{\sharp}$ to guarantee the existence of the required filters.

(2) and (3) are easy applications of Lemma \ref{CardMethod}.

(4) follows from Lemma \ref{mainlemma} below.


\end{proof}

It is, on the other hand, not hard to see that e.g. the construction problem of taking a ring to its quotient field is effectively solvable as in \cite{Ho2}.
The intuitions captured by Hodges' approach are hence preserved in our framework.



There are certainly various interesting questions to be asked about the effectivity, or otherwise, of various construction problems or $\Pi_{2}$-statements. However,
we want to take the analogy with Turing computability a bit further: Instead of merely asking what problems are solvable, we want to consider what problems/statements
are effectively reducible to which others in the sense that, given access to a solution to one as an `oracle', one can effectively solve the other. A quite straightforward way to make this idea
precise is the following:

\begin{defini}
Assume that the OTM is equipped with an extra `miracle tape'. Let $F$ be a class function taking sets or ordinals to sets of ordinals. 
An miracle-OTM-program is defined like an OTM-program, but with an extra `miracle' command. When this command is carried out, the set $X$ of ordinals
on the miracle tape is replaced by $F(X)$. We write $P^{F}$ to indicate that $P$ is run and whenever the miracle command is applied to $X$, it is replaced by $F(X)$.\footnote{We thus make the implicit assumption that the miracle tape behaves deterministically, i.e. that,
 whenever the miracle command is applied to some $X$, the outcome will be the same. However, this property is not used anywhere in the arguments below. One may thus drop it, at the price of some extra formal complications.}
\end{defini}

\begin{defini}{\label{reducibility}}
Let $C_{1}$ and $C_{2}$ be construction problems. Then $C_{1}$ is reducible to $C_{2}$, written $C_{1}\leq C_{2}$ if and only if there is some miracle-OTM-program $P$ such that the following holds:
Whenever $F$ is a canonification of $C_{2}$ and whenever $G:V\rightarrow V$ is a class function taking each code for a set $x$ to some code for $F(y)$ and $x$ is a set and $c$ a code for $x$,
we have $P^{G}(c)\downarrow=d$, where $d$ is a code for $F(x)$.
\end{defini}

\textbf{Remark}: Note that we do not demand in the conditions on $G$ that $G(c)$ depends only on $x$ when $c$ is a code for $x$. By demanding that the same reduction works for every $G$, 
we rule out the possibility of coding extra information into the input representations.

\smallskip

Concerning this notion of reducibility, we observe that certainly a cardinality-raising construction is not reducible to one that is not:

\begin{lemma}{\label{cardraising}}
Let $C_{1}$, $C_{2}$ be construction problems. Assume that there are some canonification $F$ of $C_{2}$ and some infinite set $x$ such that, for all sets $y$,
(1) if $C_{1}(x,y)$, then $|y|>|x|$ and (2) if $y$ is infinite, then $|F(y)|\leq |y|$. Then $C_{1}\nleq C_{2}$.
\end{lemma}
\begin{proof}
As in the proof of Lemma \ref{CardMethod} above, OTM-computable functions cannot raise cardinalities. By assumption, the miracle operation will also not raise the cardinality. Hence the output
of a program $P$ with a $C_{2}$-miracle will (for infinite input) always have at most the cardinality of the input and thus cannot in any case witness $C_{1}$.
\end{proof}

\textbf{Remark}: In particular, the construction problem of taking a valued field to its linear compactifications (see \cite{Ho2}, Theorem 4.10) is not reducible to any of the following construction problems:
Field to algebraic closure, formally real field to its real closure, field of characteristic $p$ to its separable algebraic closure. 

\smallskip



The above captures the idea that one construction `helps' carrying out another. There is also a much more restrictive intuitive notion of reducibility between problems, namely
that instances of one (construction) problem can be effectively `translated' to particular instances of another: Given an instance of a problem $C_{1}$, we can first effectively turn
it into an instance of a problem $C_{2}$ and then effectively turn the solution to $C_{2}$ into a solution to $C_{1}$. Another way to view this is that $C_{2}$ may only be used once
in solving $C_{1}$. Thus, we define:

\begin{defini}{\label{generalizedWeihrauch}}
Let $C_{1}$, $C_{2}$ be construction problems. Then $C_{1}$ is generalized Weihrauch reducible to $C_{2}$, written $C_{1}\leq_{\text{gW}}C_{2}$, if and only if there are OTM-programs $P$ and $Q$
such that the following holds for all sets $x$ in the domain of $C_{1}$, every code $c$ for $x$ and every canonification $F$ of $C_{2}$:
\begin{enumerate}
\item $Q(c)$ converges to output $c^{\prime}$, where $c^{\prime}$ is a code for a set $y$ 
\item For every code $c^{\prime\prime}$ of $F(y)$, $P(c^{\prime\prime})$ converges to output $c^{\prime\prime\prime}$, where $c^{\prime\prime\prime}$ is a code for a set $z$
\item We have $C_{1}(x,z)$
\end{enumerate}
If these clauses hold, we say that $(P,Q)$ witnesses the gW-reducibility of $C_{1}$ to $C_{2}$. Also, when $F$ is a canonification, $P$ and $Q$ are OTM-programs and
$x$ is a set, we write $[P,F,Q](x)$ for the $z$ obtained by the procedure just described.

If $C_{1}\leq_{\text{gW}}C_{2}$ and $C_{2}\leq_{\text{gW}}C_{1}$, we write $C_{1}\equiv_{\text{gW}}C_{2}$.
\end{defini}

\textbf{Remark}: The name of the notion is due to its obvious resemblance with Weihrauch reducibility, which is an analogous notion for classical computability. For some results on
classical Weihrauch reducibility, see e.g. \cite{BGM}.

\smallskip

We note that reducibility notions satisfy the general order-theoretic properties of reducibility relations:

\begin{lemma}{\label{orderproperties}}
Both $\leq$ and $\leq_{\text{gW}}$ are transitive and reflexive. Consequently, $\equiv_{\text{gW}}$ and $\equiv$ are reflexive, transitive and symmetric, i.e. equivalence relations.
\end{lemma}
\begin{proof}
Reflexivity is trivial, as is transitivity for $\leq$. To see that $\leq_{\text{gW}}$ is transitive, let $C_{1}$, $C_{2}$ and $C_{3}$ be construction problems such that $C_{1}\leq_{\text{gW}}C_{2}\leq_{\text{gW}}C_{3}$,
and let $(P_{i},Q_{i})$ witness the gW-reducibility of $C_{i}$ to $C_{i+1}$, for $i\in\{1,2\}$. 
Let $P_{1}\circ P_{2}$ denote the OTM-program that first carries out $P_{1}$ and then runs $P_{2}$ on the output, and define $Q_{2}\circ Q_{1}$ likewise. We claim that $(P_{2}\circ P_{1},Q_{1}\circ Q_{2})$ witnesses
the gW-reducibility of $C_{1}$ to $C_{3}$. Let $F$ be a canonification of $C_{3}$. By definition of $Q_{1}$ and $P_{2}$, $[Q_{1},F,P_{2}]$ is a canonification
of $C_{2}$. By definition of $Q_{2}$ and $P_{1}$ then, $[Q_{2},[Q_{1},F,P_{2}],P_{1}]$ is a canonification of $C_{1}$. But it is easy to see that $[Q_{2},[Q_{1},F,P_{2}],P_{1}]=[Q_{2}\circ Q_{1},F,P_{2}\circ P_{1}]$.
\end{proof}

\begin{defini}{\label{equivalence}}
Let $C$ be a construction problem. Then $[C]$ denotes the $\equiv$-equivalence class of $C$ and $[C]_{\text{gW}}$ denotes the $\equiv_{\text{gW}}$-equivalence class of $C$.
\end{defini}


\section{A Method for negative Results}

We develop a method for showing that a construction problem is not gW-reducible to another. We will work with class-sized models of ZF$^{-}$, which denotes Zermelo-Fraenkel set theory without the axiom
of powerset; more precisely, we take the formulation of ZF$^{-}$ given in \cite{GH}.

\smallskip

\textbf{Remark}: Note that the following theorem is not trivial even when $\text{ZF}^{-}$ is strengthened to full $\text{ZF}$, since a ZF model $M$ may 
contain a set $x$ without containing a suitable input format for $x$, so that the computation of an OTM cannot be simulated within $M$.

\begin{lemma}{\label{collapseset}}
Let $M\models\text{ZF}^{-}$ be transitive and suppose that $x\in M$. Then $\mathbb{P}_{x}:=\{f:\omega\rightarrow x: |f|<\omega\wedge f \text{injective}\}$ is a set in $M$.
\end{lemma}
\begin{proof}
Let $y:=x\times\omega$. For each $n\in\omega$, we have $y^{n}\in M$ and the function $F:\omega\rightarrow M$ that maps $n$ to $y^{n}$ is definable in $M$. By replacement and union,
$A:=\bigcup\{y^{n}:n\in\omega\}\in M$. Now $P_{x}$ can be obtained from $A$ via separation.
\end{proof}

\begin{thm}{\label{closedness}}
Let $F$ be a computable class function, $M\models$ZF$^{-}$ transitive such that $\text{On}^{M}=\text{On}$. Assume moreover that $x\in M$ is such that there are (in $V$) two mutually 
generic $\mathbb{P}_{x}$-generic filters $G_{1}$ and $G_{2}$ over $M$. Then $F(x)\in M$. 
\end{thm}
\begin{proof}
Let $P$ be a program witnessing the computability of $F$. 
Let $x\in M$ be as in the assumption of the Theorem. By passing to $\text{tr}(x)$ if necessary, we may assume without loss of generality that $x$ is transitive.
Let $G_{1},G_{2}$ be mututally $M$-generic filters over $\mathbb{P}_{x}$ which exist by assumption. 
In $M_{1}$ and $M_{2}$, $x$ is well-ordered in order type $\alpha$ by $\bigcup{G_{1}}$ and $\bigcup{G_{2}}$, respectively.
Hence both $M[G_{1}]$ and $M[G_{2}]$ contain tape contents coding $x$ and thus
contain the computations of $P$ on these inputs. As ZF$^{-}$ models, $M[G_{1}]$ and $M[G_{2}]$ contain the decoding of every tape content they contain. Thus $F(x)\in M[G_{1}]\cap M[G_{2}]$. As
$G_{1}$ and $G_{2}$ are mutually generic, we have $M[G_{1}]\cap M[G_{2}]=M$, so $F(x)\in M$, as desired. 
\end{proof}

\textbf{Remark}: Again, some condition on the height of $M$ is required to ensure that the convergence of programs is absolute between $V$ and $M$. In particular, a parameter-free OTM can run for
more than $\alpha$ many steps, where $\alpha$ is minimal such that $L_{\alpha}\models\text{ZF}^{-}$.

\medskip

This suggests a general method for proving, given constructions $C_{1}$ and $C_{2}$, that $C_{1}\not\leq_{\text{gW}}C_{2}$. In general, find a class $A$ sufficiently closed under OTM-computability and a canonification $F$ of $C_{2}$
such that there is some $x\in A$ with the property that the closure of $F[A]$ under OTM-computability does not contain a $C_{1}$-solution for $x$. By Theorem \ref{closedness}, we can take for $A$ a transitive
class model $M$ of ZF$^{-}$. 
We summarize the most important special case of this method in the following lemma:


\begin{lemma}{\label{mainlemma}}
Let $C_{1}$, $C_{2}$ be construction problems. Assume that there are a canonification $F$ of $C_{2}$ and a transitive class-sized $M\models\text{ZF}^{-}$ and some $x\in M\cap\text{dom}(C_{1})$ such that $M$ is closed under $F$,
but $\{y:C_{1}(x,y)\}\cap M=\emptyset$. Assume moreover that $x$ is such that there are (in $V$) two mutually generic $\mathbb{P}_{x}$-generic 
filters $G_{1}$ and $G_{2}$ over $M$. Then $C_{1}\nleq_{\text{gW}}C_{2}$.
\end{lemma}
\begin{proof}
Assume otherwise, and let $P$ and $Q$ be OTM-programs such that $(P,Q)$ witnesses the gW-reducibility of $C_{1}$ to $C_{2}$. Pick $F,M$ and $x$ as in the statement of the Lemma. Then $Q$ computes, for every code of $x$ as an input,
a code for some (unique) set $y$. By Theorem \ref{closedness}, we have $y\in M$. As $M$ is closed under $F$, we have $F(y)\in M$. Now, for every code of $F(y)$ as an input, $P$ computes a code for some (unique) set $z$. Again by Theorem 
\ref{closedness}, $z\in M$. Also, by the choice of $P$ and $Q$, we have $C_{1}(x,z)$. So $z\in \{y:C_{1}(x,y)\}\cap M$, so the latter is not empty, contradicting our assumptions.
\end{proof}





\section{Results on Generalized Effective Reducibility}

As a sample application of the notions and methods developed above, we consider variants of the axiom of choice with respect to effective reducibility.

\begin{defini}{\label{choiceprinciples}}
Denote by AC the statement that for all sets $x$, there is a function $f$ such that $f(\emptyset)=\emptyset$ and for $y\in x$, if $y\neq\emptyset$, then $f(y)\in y$.
Denote by AC$^{\prime}$ the statement that for all sets $x$ whose elements are non-empty and mutually disjoint, there is a set $r$ such that $|r\cap y|=1$ for all $y\in x$.
Denote by WO the well-ordering principle, i.e. the statement that for every set $x$, there is an ordinal $\alpha$ and a bijection $f:\alpha\leftrightarrow x$.
Finally, denote by ZL Zorn's lemma, i.e. the statement that, for every partially ordered set $(X,\leq)$ in which every ascending chain has an upper bound, there is
a $\leq$-maximal element in $X$.
\end{defini}

It is not hard to see that all of these principles are equivalent in the sense of reducibility: The usual equivalence proofs explain, modulo a transfinite version of
Church's thesis, how each of these principles can be reduced to any other. This is perhaps not entirely obvious for WO$\leq$AC, as the reduction seems to 
require a choice function for the power set of a given set and the power set of a set $x$ is not OTM-computable from $x$ (e.g. by Lemma \ref{cardraising}). We give the proof as an example.

\begin{prop}{\label{woredac}}
WO$\leq$AC
\end{prop}
\begin{proof}
Given a set $x$ and a solution $F$ for AC, construct $\alpha\in\text{On}$ along with a bijection $f:\alpha\leftrightarrow x$ recursively as follows: To begin with, set $x_{0}=x$ and $f_{0}=\emptyset$.
In the $\iota$th step, apply $F$ to $\{x_{\iota}\}$ to get some $y_{\iota}\in x_{\iota}$. 
Let $f_{\iota+1}=f_{\iota}\cup\{(\iota,y_{\iota})\}$, $x_{\iota+1}=x_{\iota}\setminus\{y_{\iota}\}$. At a limit stage $\lambda$, let $x_{\lambda}=\bigcap_{\iota<\lambda}x_{\iota}$ and $f_{\lambda}=\bigcup_{\iota<\lambda}f_{\iota}$.
Once $x_{\iota}=\emptyset$ (which must eventually happen, as $x$ is a set), stop the construction and return $f_{\iota}$, which will be a bijection between $\iota$ and $x$. This procedure can be carried out on an OTM equipped with $F$.
\end{proof}

The picture becomes much more interesting when we turn to gW-reducibility. In fact, 
we can use Lemma \ref{mainlemma} to show that the well-ordering principle is not generalized Weihrauch reducible to the axiom of choice:

\begin{thm}{\label{nonreducible}}
WO$\not\leq_{\text{gW}}$AC.
\end{thm}
\begin{proof}
(Sketch) We use Lemma \ref{mainlemma}. In Theorem D.-A.C. of \cite{Z}, it is shown how to construct a transitive model of ZF$^{-}+$AC$+\neg$WO as a union of an ascending chain
of symmetric extensions of a transitive ground model $M$ of ZF$^{-}$. Starting with $M=L$, it is easily checked that, under the assumption that $0^{\sharp}$ exists,
the construction leads to a definable transitive class model $N$ of ZF$^{-}$+AC such that some set $A\in N$ that is non-wellorderable in $N$ is countable in $V$ and moreover
$\mathbb{P}_{A}$ is countable and thus has two mutually generic filters over $N$. Hence the assumptions of Lemma \ref{mainlemma} are satisfied and the non-reducibility follows.


\end{proof}

Many of the other relations between choice principles are effective, however:





\begin{thm}{\label{reducible}}
(1) $\text{AC}^{\prime}\equiv_{\text{gW}}\text{AC}\leq_{\text{gW}}\text{ZL}$

(2) $\text{ZL}\leq_{\text{gW}}\text{WO}$
\end{thm}

\begin{proof}

The proofs consists in checking that the usual equivalence proofs over ZF in fact effectivize. This is trivial for (1). We give some detail on (2) as an exemplary case.

(2) 
For ZL$\leq_{\text{gW}}$WO, let $(x,\leq)$ be a partially ordered set satisfying the assumptions of ZL. Let $Q$ be an OTM-program that, given a code $c((x,\leq))$ for $x$ on the input tape, copies $c(x)$ to
the miracle tape. After applying any canonification for WO, the miracle tape will contain a code $c^{\prime}$ for a well-ordering $<_{x}$ of $x$. Now let $P$ be an OTM-program that runs as follows: Given an (initially empty)
set $X$ of elements of $x$ on the scratch tape, compute $x\setminus X$ and search through it for the $<_{x}$-minimal element $e$ great than all elements of $X$. If none exists, return $e$, otherwise
set $X=X\cup\{e\}$ and continue. This computes a maximal element of $(x,\leq)$, so $(P,Q)$ witnesses ZL$\leq_{\text{gW}}$WO.
\end{proof}

\textbf{Remark}: We do not know whether ZL belongs to one of the gW-degrees $[\text{AC}]_{\text{gW}}$, $[\text{WO}]_{\text{gW}}$, is reducible to AC, lies strictly in between or is incompatible with AC.
We suspect that ZL$\nleq_{\text{gW}}$AC. Our current state of knowledge
hence gives some meaning to the humorous claim that `The Axiom of Choice is obviously true, the well-ordering principle obviously false, and who can tell about Zorn's lemma?'.

\section{Conclusion and Further Work}

We have introduced notions of effectivity, reducibility and `case-wise' reducibility applicable to mathematical objects of arbitrary cardinality. The approach to effectivity is supported by the remarkable conceptual stability of
ordinal computability (see e.g. \cite{Fi} or \cite{Ca}) and moreover, while not equivalent to e.g. the approach by Hodges, agrees with it concerning the results obtained so far. With regard to reducibility, we have seen
how set-theoretical techniques can be used to distinguish between various versions of set-theoretical principles usually regarded as equivalent.

Clearly, there is a host of questions asking which statements are effectively reducible or gW-reducible to which others. This may be viewed as a cardinality-independent version of reverse mathematics (as e.g. considered in \cite{Sh}) 
and the theory of the Weihrauch lattice. Apart from that, it may be interesting to consider variants of these notions with parameter-free computability replaced by other models of transfinite computation, like 
Infinite Time Turing Machines (\cite{HL}) or OTMs with ordinal parameters. Another worthwhile topic would be to replace (relativized) computability with (relativized) recognizability (see e.g. \cite{CSW}). 

Finally, various notions from classical computability theory could be incorporated into our framework: For example, one should be able to make sense of the concept of a `random construction' and ask
whether there are interesting non-effective constructions that are reducible to them. We will also consider candidates for a sensible notion of a `jump operator' for construction problems, a notion
that led to a number of fascinating results about Weihrauch reducibility (\cite{BGM}).

\end{document}